\pdfoutput=1
\RequirePackage{ifpdf}
\ifpdf 
\documentclass[pdftex]{sigma}
\else
\documentclass{sigma}
\fi

\numberwithin{equation}{section}

\newtheorem{Theorem}{Theorem}[section]
\newtheorem{Lemma}[Theorem]{Lemma}
\newtheorem{Proposition}[Theorem]{Proposition}
{ \theoremstyle{definition}
\newtheorem{Example}[Theorem]{Example} }

\begin{document}

\allowdisplaybreaks

\renewcommand{\thefootnote}{$\star$}

\newcommand{\arXivNumber}{1512.03898}

\renewcommand{\PaperNumber}{050}

\FirstPageHeading

\ShortArticleName{Automorphisms of Algebras and Bochner's Property for Vector Orthogonal Polynomials}

\ArticleName{Automorphisms of Algebras and Bochner's Property\\ for Vector Orthogonal Polynomials\footnote{This paper is a~contribution to the Special Issue
on Orthogonal Polynomials, Special Functions and Applications.
The full collection is available at \href{http://www.emis.de/journals/SIGMA/OPSFA2015.html}{http://www.emis.de/journals/SIGMA/OPSFA2015.html}}}

\Author{Emil HOROZOV~$^{\dag\ddag}$}

\AuthorNameForHeading{E.~Horozov}

\Address{$^\dag$~Department of Mathematics and Informatics, Sofia University,\\
\hphantom{$^\dag$}~5 J.~Bourchier Blvd., Sofia 1126, Bulgaria}
\Address{$^\dag$~Institute of Mathematics and Informatics,	Bulg. Acad. of Sci.,\\
\hphantom{$^\ddag$}~Acad. G.~Bonchev Str., Block 8, 1113 Sofia, Bulgaria}
\EmailD{\href{mailto:horozov@fmi.uni-sofia.bg}{horozov@fmi.uni-sofia.bg}}

\ArticleDates{Received January 26, 2016, in f\/inal form May 12, 2016; Published online May 19, 2016}

\Abstract{We construct new families of vector orthogonal polynomials that have the pro\-per\-ty to be eigenfunctions of some dif\/ferential operator. They are extensions of the Hermite and Laguerre polynomial systems. A third family, whose f\/irst member has been found by Y.~Ben Cheikh and K.~Douak is also constructed. The ideas behind our approach lie in the studies of bispectral operators. We exploit automorphisms of associative algebras which transform elementary vector orthogonal polynomial systems which are eigenfunctions of a~dif\/ferential operator into other systems of this type.}

\Keywords{vector orthogonal polynomials; f\/inite recurrence relations; bispectral problem; Bochner theorem}

\Classification{34L20; 30C15; 33E05}

\renewcommand{\thefootnote}{\arabic{footnote}}
\setcounter{footnote}{0}

\section{Introduction}\label{intro}

Salomon Bochner \cite{Bo} has classif\/ied all systems of orthogonal polynomials $P_n(x)$, $n=0, 1, \ldots $ (with respect to some measure on the real line) that are also eigenfunctions of a second-order dif\/ferential operator
\begin{gather} \label{BP1}
	L(x, \partial_x) = A(x)\partial_x^2 + B(x)\partial_x + C(x)
\end{gather}
with eigenvalues $\lambda(n)$.
Here the coef\/f\/icients $A$, $B$, $C$ of the dif\/ferential equation do not depend on the index (degree)~$n$ of the polynomial~$P_n(x)$.
The orthogonality condition, due to a classical theorem by Favard--Shohat is equivalent to the well known 3-terms recursion relation
\begin{gather*}
xP_n = P_{n+1} + \beta(n)P_n + \gamma(n) P_{n-1},
\end{gather*}
where $\beta(n), \gamma(n) > 0$ are constants, depending on $n$.\footnote{Here we use polynomials, normalized by the condition that the coef\/f\/icient of their highest-order term is~$1$.} Bochner's theorem states that all polynomial systems with such properties are the classical orthogonal polynomials of Hermite, Laguerre and Jacobi. From now on we will replace the condition of orthogonality with respect to a measure with orthogonality with respect to a non-degenerate functional on the space of the polynomials of one variable ${\mathbb{C}}[x]$ with complex coef\/f\/icients. In this setting we have to add to the classical orthogonal polynomials also the Bessel polynomials.

Generalizations of Bochner's result were suggested long ago by H.L.~Krall \cite{Kra}. He classif\/ied all order 4 dif\/ferential operators which have a family of orthogonal polynomials as eigenfunctions. Later many contributions in this direction were made by T.~Koornwinder~\cite{Koo}, L.~Littlejohn~\cite{Lit}, J.~Koekoek, R.~Koekoek~\cite{KK}, etc.

In recent times there is much activity in generalizations and versions of the classical result of Bochner, see, e.g., \cite{DdlI, GH1, GH2, Il1, Il2}.
An important role in some of these generalizations is played by ideas from the study of the bispectral problem, initiated in \cite{DG}. Even translating the results of Bochner and Krall into this language already gives a good basis to continue investigations~\cite{GH2}. It goes as follows. Let us introduce the function $\psi(x, n) = P_n(x)$. If we write the right-hand side of the 3-term recursion relation as a dif\/ference operator~$\Lambda(n)$ acting on functions $f(n)$ of the discrete variable~$n$ as
\begin{gather*}
\Lambda(n)f(n)= f(n+1) + \beta(n)f(n) +\gamma(n)f(n-1),
\end{gather*}
then the 3-term recursion relation can be written as
\begin{gather}
	x\psi(x, n) = \Lambda(n) \psi(x, n). \label{BP2a}
\end{gather}
This means that $\psi(x, n)$ is an eigenfunction of the discrete operator $\Lambda$ with eigenvalue $x$ and of a dif\/ferential operator $L(x, \partial_x)$ with eigenvalue~$\lambda_n$. Hence we can formulate the Bochner--Krall problem as

{\it Find all systems of polynomials $P_n(x)$ that are eigenfunctions of a second-order difference ope\-ra\-tor~$\Lambda(n)$ \eqref{BP2a}, where $\gamma(n) \neq 0$, with eigenvalue $x$ and of a differential operator $L(x, \partial_x)$}~\eqref{BP1} with eigenvalues~$\lambda(n)$.

In \cite{GHH, GY} and several other papers the authors make use of Darboux transformations to construct families of systems of orthogonal polynomials, that are eigenfunctions of even-order dif\/ferential operators. Some properties of these polynomial systems are dif\/ferent from the pro\-per\-ties of the classical orthogonal polynomials. For example P.~Iliev \cite{Il1, Il2} has studied the algebras of operators that have Bochner--Krall polynomials as eigenfunctions. These algebras have two or more generators, whereas the algebras corresponding to the classical OP are isomorphic to the algebra~${\mathbb{C}}[X]$.

Here we also use ideas found in the studies of bispectral operators but of dif\/ferent nature. Before explaining these as well as the main results let us introduce one more concept which is central for the present paper. This is the notion of vector orthogonal polynomials (VOP), introduced by J.~van Iseghem~\cite{VIs}. Let $\{P_n(x)\}$ be a family of monic polynomials such that $\deg P_n = n$. Assume that they satisfy a~$(d+2)$-term recursion relation, $d \geq 1$
\begin{gather}
xP_n(x) = P_{n+1} + \sum_{j=0}^{d}\gamma_j (n)P_{n-j}(x) \label{d-ort}
\end{gather}
with coef\/f\/icients $\gamma_j(n)$ independent of~$x$ and $\gamma_d(n) \neq 0$. Then by a theorem of P.~Maroni \cite{Ma} there exist~$d$ functionals $u_j$, $j =0,\ldots,d-1$ on the space of all polynomials ${\mathbb{C}}[x]$ such that
\begin{gather*}
u_k (P_nP_m) = 0, \qquad m > nd + k, \qquad n \geq 0, \\
u_k(P_nP_{nd+k}) \neq 0, \qquad n \geq 0,
\end{gather*}
for each $k \in N_{d+1} := \{0, \ldots, d-1 \}$. When $d = 1$ this is the notion of orthogonal polynomials.

In the last 20--30 years there has been much activity in the study of vector orthogonal polynomials and the broader class of multiple orthogonal polynomials (see the cited references below).

Applications of $d$-orthogonal polynomials include the simultaneous Pad\'e
approximation problem where the multiple orthogonal polynomials appear~\cite{Apt,ABVA, ApKu, AMR, DBr}. Multiple orthogonal polynomials play an important role in random matrix theory~\cite{ApKu, BDK}. Applications to integrable systems and quantum mechanics are found as well~\cite{Cha, DLMTC}.

One problem that deserves attention is to f\/ind analogs of classical orthogonal polynomials. Several authors~\cite{KR, VAC} have found multiple orthogonal polynomials, that share a number of properties with the classical orthogonal polynomials~-- they have raising and lowering operators, Rodrigues type formulas, Pearson equations for the weights, etc. However one of the features of the classical orthogonal polynomials~-- a~dif\/ferential operator for which the polynomials are eigenfunction, is missing. Sometimes this property is relaxed to the property that the polynomials satisfy linear dif\/ferential equation, whose coef\/f\/icients may \textit{depend} on the degree of the polynomial.

The ideas from bispectral theory that we use, allow us to to construct new VOP systems, having the strict property to be eigenfunctions of one dif\/ferential operator. This is even stronger than one of the research problems suggested in \cite{VAC}. We notice that Hahn's characterization, while valid for the polynomials of Hermite type (Appell polynomials) is not shared by the new VOP, constructed in Sections~\ref{section4} and~\ref{section5} (this was guessed in~\cite{VAC}). In these cases the relevant property is a generalization of Shef\/fer's one: There exist an operator $Q(x, \partial)$ such that
\begin{gather*}
Q(x, \partial) P_n(x) = \mu(n) P_{n-1}.
\end{gather*}
Recall that Shef\/fer's property corresponds to $\mu(n) = n$ and $Q = \partial$. But in the constructions in the present paper~$\mu(n)$ is a~polynomial of degree higher than one and $Q$ can be a higher-order dif\/ferential operator.

We are looking for polynomials $P_n(x)$, $n=0, 1, \ldots$ that are eigenfunctions of a dif\/ferential operator $L$ with eigenvalues depending on the discrete variable $n$ (the index) as in~\eqref{BP1} and which at the same time are eigenfunctions of a dif\/ference operator in $n$, i.e., f\/inite-term recursion relation~\eqref{d-ort} with an eigenfunction, depending only on the continuous variable $x$. This is a~$(d+2)$-order recursion relation with some $d \geq 1$. Systems with these properties will be called \textit{polynomial systems with Bochner's property.}

One of our results includes an extension of Laguerre polynomials. We construct systems of vector orthogonal polynomials $\{P_n(x)\}$ which are eigenfunctions of a dif\/ferential operator starting from a very simple bispectral situation. Our approach uses ideas of the bispectral theory from \cite{BHY} but does not use Darboux transformations, which is usually the case, see, e.g., \cite{GH3, GHH, GY, Il1, Il2}. We use methods introduced in~\cite{BHY} which we describe below. Also a well known extension of Hermite polynomials (see~\cite{App, Dou, Sh}) is presented. The reason to repeat this well known result is that {\it our construction is a new one}. Also the case of Hermite-like polynomials is the simplest and illustrates the idea without going to computations.

Here is a sketch of the method explained on the example of Hermite-like polynomials. We start with the Weyl algebra $W_1$, i.e., the algebra of dif\/ferential operators in one variable $x$ with polynomial coef\/f\/icients. Another algebra~${\mathcal R}$ is spanned by the shift operator $T$ in $n$, $Tf(n) = f(n+1)$, its inverse~$T^{-1}$ and~$n$. There is a very simple example of the bispectral problem with $\psi(x, n) = x^n$. It is a joint eigenfunction of the operators $H= x\partial_x$ and the shift operator~$T$. In more details
\begin{gather*}
H\psi(x,n) = n \psi(x,n), \qquad T \psi(x,n) = x \psi(x,n).
\end{gather*}
This def\/ines an anti-involution $b\colon W_1 \rightarrow {\mathcal R}$, $b(H) =n$, $b(T) = x$, $b(\partial)=nT^{-1}$. If we perform an automorphism of the Weyl algebra $\sigma\colon W_1 \rightarrow W_1$ we can obtain another anti-isomorphism $b' = b\circ \sigma^{-1}\colon W_1 \rightarrow {\mathcal R}$. This will def\/ine a new bispectral problem by taking $H' = b\circ \sigma(H)$ and $\Lambda = b\circ \sigma^{-1}(x)$. In the case described brief\/ly here we obtain Hermite polynomials from the simplest nontrivial automorphism $\sigma = e^{\operatorname{ad}_B}$ with $B =-\partial^2/2$. Using other automorphisms we obtain systems of vector orthogonal polynomials in general, which are known under the name Appell polynomial systems \cite{App, Dou, Sh}.

Instead of~$W_1$ we can use other algebras. For example we can take a suitable enveloping algebra of~${\mathfrak{sl}}_2$. This will give us Laguerre polynomials and some of their vector orthogonal generalizations. Using a dif\/ferent algebra in Section~\ref{section5}, we come to new classes of VOP. In fact its f\/irst member is also known and can be found in~\cite{BCD, BCO}.

In \cite{VZh} the authors use automorphisms of the Heisenberg--Weyl algebra to construct vector orthogonal polynomials of Hermite and Charlier type. A nice feature of their construction is that a connection to representation theory is underlined and heavily used. Our construction uses only automorphism of algebras as def\/ined in~\cite{BHY} and seems to be simpler and easier to apply elsewhere.

Here we keep the exposition as simple as possible. We do not study properties of the polynomial systems and leave this for future work.
The construction of the present paper can be performed in much more general situations of VOP, including their discrete versions, see~\cite{Ho}, matrix orthogonal polynomials, multivariable orthogonal polynomials, etc. This will be done elsewhere.

\section{Elements of bispectral theory}\label{section2}

The following introductory material is mainly borrowed from \cite{BHY}. Here we present it in a form suitable for the continuous-discrete version of the bispectral problem.

For $i=1,2$, let $\Omega_i$ be two open subsets of ${\mathbb{C}}$ such that $\Omega_2$ is invariant
under the translation operator
$T \colon n \mapsto n+1$ and its inverse $T^{-1}$.

A complex analytic dif\/ference operator on $\Omega_2$
is a f\/inite sum of the form
\begin{gather*}
\sum_{k \in {\mathbb Z}} c_k(n) T^k,
\end{gather*}
where $c_k \colon \Omega_1 \to {\mathbb{C}}$ are analytic functions.

By ${\mathcal B}_1$ we denote an algebra with unit, consisting of dif\/ferential operators $L(x, \partial)$ in one variable $x$. By
${\mathcal B}_2$ we denote an algebra of dif\/ference operators $\Lambda(n, T)$ with unit. Denote by~${\mathcal M}$ the space of complex analytic functions
on $\Omega_1 \times \Omega_2$. The space ${\mathcal M}$ is naturally equipped with the structure of bimodule
over the algebra of analytic dif\/ferential operators $L(x, \partial_x)$ on $\Omega_1$ and the dif\/ference operators on~$\Omega_2$.

Assume that there exists an algebra isomorphism $b \colon {\mathcal B}_1 \to {\mathcal B}_2$ and
an element $\psi \in {\mathcal M}$ such that
$P \psi = b(P)\psi$, $\forall\, P \in {\mathcal B}_1$.
Assume that $A_i$ and $K_i$ are subalgebras of ${\mathcal B}_i$ such that
$b(A_1) = K_2$ and $b(K_1) = A_2$.
In our case $K_1$ will be an algebra of polynomials in $x$ and~$K_2$ will be an algebra of polynomials in~$n$.

A continuous-discrete bispectral function is by def\/inition an element of~${\mathcal M}$ (i.e.,
an analytic function) $\psi \colon \Omega_1 \times \Omega_2 \to {\mathbb{C}}$
for which there exist analytic dif\/ferential operator
$L(x, \partial_x)$ on~$\Omega_1$, analytic dif\/ference opera\-tor~$\Lambda(n, T)$ on~$\Omega_2$, and analytic functions
$\theta(x)$ and $\lambda(n)$,
such that
\begin{gather}\label{bisp3}
		L(x, \partial_x) \psi(x,n) = \lambda(n) \psi(x,n), \qquad
		\Lambda(n, T_n) \psi(x, n) = \theta(x) \psi(x,n)
	\end{gather}
on $\Omega_1 \times \Omega_2$. In fact, as we would be interested in VOP, we will consider only the case when $\theta(x) \equiv x$. We will assume that $\psi(x,n)$ is a nonsplit function of~$x$ and~$n$ in the sense that it satisf\/ies the condition
\begin{enumerate}\itemsep=0pt
\item[$(**)$] there are no nonzero analytic dif\/ference operators $L(x, \partial_x)$ and $\Lambda(n, T)$
that satisfy one of the above conditions with $f(n) \equiv 0$ or $\theta(x) \equiv 0$.
\end{enumerate}

The assumption $(**)$ implies that the map $b \colon {\mathcal B}_1 \to {\mathcal B}_2$, given by $b(P(x, \partial_x)):= S(n, T)$ is a well def\/ined algebra anti-isomorphism. The algebra $A_1 := b^{-1}(K_2)$ consists of the bispectral operators corresponding to~$\psi(x,z)$ (i.e.,
dif\/ferential operators in $x$ having the properties~\eqref{bisp3}) and the algebra
$A_2 := b(K_1)$ consists of the bispectral operators corresponding to~$\psi(x,n)$ (i.e.,
dif\/ference operators in $n$ having the properties~\eqref{bisp3}).

Below we present the dif\/ferential-dif\/ference version of the general bispectral problem, suitable in the set-up of vector orthogonal polynomial sequences, which are eigenfunctions of dif\/ferential operators.
We are interested in the case when, for any f\/ixed~$n$, the function $\psi(x,n)$ def\/ining the map~$b$ is a polynomial in $x$. We additionally assume that, for all~$n$, polynomials $\psi(x,n)$ are the eigenfunctions of a f\/ixed dif\/ferential operator in the variable~$x$ and that, for any f\/ixed~$x$, the function $\psi(x,n)$ is an eigenfunction of a~dif\/ference operator in~$n$. We know that such a situation occurs in case of the classical orthogonal polynomials.

Let ${\mathcal B}_1$ be the Weyl algebra $W_1$ (spanned over~${\mathbb{C}}$ by~$x$ and~$\partial$). To def\/ine another algebra, we introduce three operators: the shift operator $T$ acting on functions~$f(n)$ of the discrete variable $n$ by $Tf(n) := f(n+1)$, its inverse~$T^{-1}$ shifting in the opposite direction, i.e., $T^{-1}f(n) : = f(n-1)$, and the operator~$n$ of multiplication by the variable $n$. We def\/ine the algebra~${\mathcal R}_2$ of dif\/ference operators over ${\mathbb{C}}$, spanned by $T$, $T^{-1}$ and $n$ with the obvious relations. Let ${\mathcal M}$ be a left module over~${\mathcal B}_1$ and~${\mathcal R}_2$. In our case such a module is a linear space of bivariate analytic func\-tions~$f(x,n),$ where $x$ is a~continuous and $n$ is a discrete variables.

Set $L := x\partial$. Introducing $\psi(x,n): = S_n(x):=x^n$, we get
\begin{gather}
	L S_n(x) = n S_n(x), \label{diff-eq}\\
	x S_n(x) = S_{n+1}(x), \label{rec-rel} \\
	\partial S_n(x) = n S_{n-1}(x). \label{low-op}
\end{gather}

We see that \eqref{diff-eq} and \eqref{rec-rel} are providing an instance of a dif\/ferential-dif\/ference bispectral pair of operators. The equation \eqref{diff-eq} guarantees that $\psi(x,n)$ is an eigenfunction of the dif\/ferential operator~$L$, and~\eqref{rec-rel} shows that the same function is an eigenfunction of the dif\/ference ope\-ra\-tor~$T$.
Equation~\eqref{low-op} follows from the previous two; it is called the dif\/ferentiation (or the lowering) formula and is quite important.

We can def\/ine the map $b\colon {\mathcal B}_1 \rightarrow {\mathcal R}_2$ as given by
\begin{gather*}
b(\partial) = nT^{-1}, \qquad b(x) = T, \qquad b(L) = n.
\end{gather*}
(It is clear that the f\/irst two identities are suf\/f\/icient to determine $b$ completely. The last identity is included because of its importance and convenience.) Then we put ${\mathcal B}_2 = b({\mathcal B}_1)$.

Next we will discuss how to construct new bispectral operators from already known ones. First recall that for an operator~$L$ it is said that $\operatorname{ad}_L$ acts locally nilpotently when for any element $a\in {\mathcal B}$ there exists $k\in {\mathbb N}$, such that
$\operatorname{ad}^k_L(a)= 0$.

The main tool for constructing bispectral operators in this paper will be the following simple observation made in~\cite{BHY}.

\begin{Proposition}\label{2.1}\sloppy
 Let ${\mathcal B}_1$, ${\mathcal B}_2$ be unital algebras with the properties described above and \mbox{$b \colon {\mathcal B}_1 \to {\mathcal B}_2$}.
Let $L \in {\mathcal B}_1$ and $\operatorname{ad}_L \colon {\mathcal B}_1 \rightarrow {\mathcal B}_1$ be a locally nilpotent operator. Suppose that, for any fixed~$n$, $e^L\psi(x,n)$ is a polynomial in~$x$ of degree~$n$. Define a new map $b' \colon {\mathcal B}_1 \rightarrow {\mathcal B}_2$ via the new polynomial function $\psi'(x,n) := e^L \psi(x,n)$.
Then $b' = b(e^{-\operatorname{ad}_L})$ and $b'\colon {\mathcal B}_1\to {\mathcal B}_2$ is a~bispectral anti-involution.
\end{Proposition}

\section{Hermite-like polynomials} \label{section3}

Observe that $\operatorname{ad}_{\partial}$ acts locally nilpotently on the algebra ${\mathcal B}_1$. The same is true when we take instead of~$\partial$ any polynomial $q(\partial)$ without a~free term. From this it follows that for any element~$a\in {\mathcal B}_1$, the series
\begin{gather*}
\sum_{k=0}^{\infty} \frac{\operatorname{ad}^k_{\partial}(a)}{k!}
\end{gather*}
has only f\/initely many terms.
For such a polynomial $q(\partial)$ we def\/ine the automorphism \mbox{$\sigma {:=} e^{\operatorname{ad}_{q(\partial)}}$}.
Its action on $L$ is explicitly given by
\begin{gather*}
\sigma(L)= L + \sum_{k=1}^{\infty} \frac{\operatorname{ad}^k_{q(\partial)}(L)}{k!} =
L + [q(\partial), L] = L + q'(\partial)\partial,
\end{gather*}
since the rest of the terms vanish.

Let us f\/irst compute the action of $\sigma = e^{\operatorname{ad}_{q(\partial)}}$ on the standard basis of ${\mathcal B}_1$. We have
\begin{gather*}
\sigma(\partial) =\partial,\qquad \sigma(x) = x + q'(\partial).
\end{gather*}

Next we def\/ine a new map $b'\colon {\mathcal B}_1\to {\mathcal B}_2$ by $b':=b\circ \sigma^{-1}$. We will introduce the new function~$\psi_1(x,n)$ def\/ined via
\begin{gather*} 
	\psi_1(x,n) := P_n^q(x): = e^{q(\partial)}x^n = x^n + \sum_{m=1}^{\infty}\frac{(q(\partial))^m x^n}{m!}.
\end{gather*}
Due to the fact that $q(\partial)$ reduces the degrees of the polynomials it is clear that the latter series contains only a~f\/inite number of terms. Hence it is a polynomial in $x$. (Here we use that $q$ has no constant term.)
Def\/ine the operator $L_1: = \sigma(L)$. We already saw that
\begin{gather*}
L_1 = x\partial + q'(\partial)\partial.
\end{gather*}

\begin{Lemma}\label{ln:new}
The new bispectral anti-involution $b'$ satisfies:
\begin{gather}\label{1-inv}
b'(\partial) = nT^{-1},\qquad b'(x) = T - q'\big( nT^{-1} \big),\qquad b'(L_1) = n.
\end{gather}
\end{Lemma}

\begin{proof} We have $b'(\partial) = b \circ\sigma^{-1} (\partial) = b (\partial) = nT^{-1}$.
In the same way we compute
$b'(x) = b \circ\sigma^{-1} (x) = b (x - q'(\partial)) = T - q'( nT^{-1} )$.
Finally for $L_1$ we f\/ind
$b'(L_1) = b\circ \sigma^{-1}(L_1) = b\circ \sigma^{-1} \circ\sigma(L) = b(L) = n$.
\end{proof}

\begin{Theorem}\label{mult-App}
The polynomials $P_n^q$ have the following properties:
\begin{enumerate}\itemsep=0pt
\item[$(i)$] they are eigenfunctions of the differential operator
\begin{gather}
	L_1 := q'(\partial)\partial + x\partial;
\end{gather}

\item[$(ii)$] they satisfy the recurrence relation
\begin{gather}
	xP^q_n(x) = P^q_{n+1} - q'\big(nT^{-1}\big) P^q_n(x); \label{new-rec}
\end{gather}

\item[$(iii)$] the following differentiation formula $($lowering operator$)$
\begin{gather*}
	\partial P^q_n(x) = nP^q_{n-1} 
\end{gather*}
 holds.
 \end{enumerate}
\end{Theorem}

\begin{proof} Indeed, statement (i) that polynomials $P_n(x)$ are eigenfunctions of the operator
\begin{gather*}
L_1 = q'(\partial)\partial + x\partial
\end{gather*}
follows from the computation of the new bispectral involution $b'$ in Lemma~\ref{ln:new}. More exactly, the third equation of~\eqref{1-inv} gives
\begin{gather*}
(q'(\partial)\partial + x\partial)P_n(x) = n P_n(x).
\end{gather*}
To prove recurrence~(ii), observe the second equation in~\eqref{1-inv} provides
\begin{gather*}
x P_n(x) = P_{n+1} - q'\big( nT^{-1} \big)P_n.
\end{gather*}
Also\looseness=1 the f\/irst equation in \eqref{1-inv} gives the dif\/ferentiation formula
$\partial P_n = nP_{n-1}$, which settles~(iii).
\end{proof}

\section{Laguerre-like VOP}\label{section4}

We start with an algebra ${\mathcal B}_1$ spanned by the operators~$x$, $H = x\partial$ and $B = x\partial^2 +\beta \partial$. The commutation relations are
\begin{gather*}
[H, x]= x, \qquad [B, x]= 2H + \beta ,\qquad [H, B]= - B.
\end{gather*}
This shows that our algebra is an enveloping algebra of $\mathfrak{sl}_2$ (not the universal one).

Next we def\/ine the algebra ${\mathcal R}_2$ of dif\/ference operators over ${\mathbb{C}}$, as the
one spanned by~$T$,~$T^{-1}$,~$n$. Let ${\mathcal M}$ be a left module over ${\mathcal B}_1$ and ${\mathcal R}_2$. In our case such a module is a linear space of bivariate functions $f(x,n)$, where~$x$ is a continuous and~$n$ is a~discrete variables. We def\/ine~$K_1$ to be the algebra of polynomials in~$x$ and~$K_2$ to be the algebra of polynomials in~$n$.

Let again $\psi(x,n):= S_n(X):= x^n$. It is obviously nonsplit in the above def\/ined sense~$(**)$. Def\/ine the map $b$ by its action on the
generators:
\begin{gather*}
b(x) = T, \qquad b(H) = n, \qquad b(B)= n(n-1+ \beta)T^{-1}.
\end{gather*}
We def\/ine the algebra ${\mathcal B}_2$ to be the image $b({\mathcal B}_1)$ of ${\mathcal B}_1$ under the map $b$. Obviously the map $\sigma\colon {\mathcal B}_1 \rightarrow {\mathcal B}_2$ is an anti-automorphism.

Setting $L:= H$ we get
\begin{gather*}
	H S_n(x) = n S_n(x), \qquad 
	x S_n(x) = S_{n+1}(x), \qquad 
	B S_n(x) = n(n-1+ \beta) S_{n-1}(x).
\end{gather*}

Using the above notation and following the scheme of the previous section, we now set up an appropriate bispectral problem.
We will use the algebras ${\mathcal B}_1$ and ${\mathcal B}_2$ introduced there.

Consider a polynomial $q(X)$. For simplicity we assume that
$q(X)$ has no constant term. Let us def\/ine an automorphism $\sigma\colon {\mathcal B}_1 \rightarrow {\mathcal B}_1$
\begin{gather*}
\sigma(A) = e^{\operatorname{ad}_{q(B)}}(A),\qquad A \in {\mathcal B}_1.
\end{gather*}
In the next lemma we compute the action of $\sigma$.

\begin{Lemma}\label{gen}
$\sigma$ acts on the generators as follows
\begin{gather*}
\sigma(x) = x + (2H + \beta)q'(B) + q''(B)B + q'(B)^2B, \\
\sigma(H) = H + q'(B)B, \qquad
\sigma(B) = B.
\end{gather*}
\end{Lemma}

\begin{proof}In order to derive the f\/irst formula we need to compute the commutator
\begin{gather*}
[B^m, x] = 2\sum_{j=0}^{m-1}B^jH B^{m-1-j} +\beta B^{m-1}.
\end{gather*}
Notice that $B^jH = (H + j)B^j$. From this we f\/ind
\begin{gather*}
[B^m, x] = 2mHB^{m-1} + (m(m-1) + \beta m)B^{m-1},
\end{gather*}
which shows that
\begin{gather*}
[q(B), x] = (2H + \beta)q'(B) + q''(B)B.
\end{gather*}
In order to compute $\operatorname{ad}^2_q(x)$ we also need $[q(B), H] = q'(B) B$,
which yields
\begin{gather*}
[q(B), Hq'(B)] = q'^2(B)B.
\end{gather*}
This gives that
\begin{gather*}
\operatorname{ad}^2_q(x) = 2q'^2(B)B, \qquad \operatorname{ad}^3_q(x) = 0.
\end{gather*}
The two other formulas are obvious.
\end{proof}

The lemma shows that $\operatorname{ad}_{q(B)}$
acts locally nilpotently on the algebra ${\mathcal B}_1$, i.e., when applied to any element $A\in {\mathcal B}_1$, the series
\begin{gather*}
\sum_{k=0}^{\infty} \frac{\operatorname{ad}_{q(B)}^k(A)}{k!}
\end{gather*}
has only f\/initely many terms.

Let us def\/ine a new map $b'\colon {\mathcal B}_1\to {\mathcal B}_2$ given by $b':=b\circ \sigma^{-1}$ and
a~new function $\psi_1(x,n)$ def\/ined via
\begin{gather*} 
	\psi_1(x,n) := P_n^q(x): = e^{q(B)}x^n = x^n + \sum_{k=1}^{\infty}\frac{(q(B))^k x^n}{k!}.
\end{gather*}
It is clear that the latter series contains only a f\/inite number of terms as the operator~$q(B)$ reduces the degrees of the polynomials. (Here we use that~$q$ has no constant term.) Hence it is a polynomial in~$x$.

Below and further in the paper we also use without mentioning that
\begin{gather*}
\sigma^{-1} = \sum_{j=0}^{\infty}\frac{(-\operatorname{ad})^j_{q(B)}}{j!}.
\end{gather*}

Def\/ine the operator $L_1: = \sigma(L)$. One can easily see that
$L_1 = x\partial + q'(B)B$.

\begin{Lemma}\label{2new-b}
The new bispectral involution $b'$ on ${\mathcal B}'$ satisfies:
\begin{gather*}
		b'(B) = n(n-1 + \beta)T^{-1}, \\
		b'(x) = T - q'(b(B))(2n+\beta) - q''(b(B))b(B) + q'^2(b(B))b(B),		 \\
		b'(L_1) = n. 
	\end{gather*}
\end{Lemma}

\begin{proof} We have $b'(B) = b \circ\sigma^{-1} (B) = b (B) = n(n-1 +\beta)T^{-1}$.
In the same way we compute
\begin{gather*}
	b'(x) = b \circ\sigma^{-1} (x) = b \big(x - (2H+\beta)q'(B) - q''(B)B + q'^2(B)B\big)\\
\hphantom{b'(x)}{} = T - q'(b(B))(2n +\beta) - q''(b(B))b(B) +q'^2(b(B))b(B).
\end{gather*}
Finally for $L_1$ we f\/ind
$b'(L_1) = b\circ \sigma^{-1}(L_1) = b\circ \sigma^{-1} \circ\sigma(L) = b(L) = n$.
\end{proof}

\begin{Theorem}
The polynomials $P_n^q$ have the following properties:
\begin{enumerate}\itemsep=0pt
\item[$(i)$] they are eigenfunctions of the differential operator
\begin{gather*}
	L_1 := q'(B)B + x\partial; 
\end{gather*}

\item[$(ii)$] they satisfy the recurrence relation
\begin{gather*}
xP^q_n(x) = P^q_{n+1} - q'\big(n(n-1+ \beta)T^{-1}\big)(2n + \beta) P^q_n(x) \\
\hphantom{xP^q_n(x) =}{} +\!\big\{ {-}2q''\big(n(n-1+ \beta)T^{-1}\big) + q'^2\big(n(n-1+ \beta)T^{-1}\big)\big\} n(n-1+ \beta) P_{n-1}^q(x); 
\end{gather*}

\item[$(iii)$] the following differentiation formula $($lowering operator$)$
\begin{gather*}
	B P^q_n(x) = n(n-1 +\beta)P^q_{n-1} 
\end{gather*}
holds.
\end{enumerate}
\end{Theorem}

\begin{proof} All the statements follow easily from the computation of the new bispectral involution~$b'$ in Lemma~\ref{2new-b}. Essentially the proof is the same as the proof for Hermite-like polynomials.
\end{proof}

\section{Automorphisms generated by third-order operators}\label{section5}

We start with an algebra ${\mathcal B}_1$ spanned by the operators $x$, $H = x\partial$ and $B = x^2\partial^3 + \alpha x \partial^2 + \beta \partial$. The commutation relations are
\begin{gather*}
[H, x]= x, \qquad [B, x]= W:=3H^2 + 2 \alpha H+ \beta,\qquad [H, B]= - B.
\end{gather*}

We again use the algebra ${\mathcal R}_2$ of dif\/ference operators over ${\mathbb{C}}$, spanned by $T$, $T^{-1}$, $n$ and the bimodule~${\mathcal M}$ of bivariate functions~$f(x,n)$, where~$x$ is a continuous and~$n$ is a discrete variables.

Let $S_n(x):=\psi(x,n):= x^n$. It is obviously nonsplit in the above def\/ined sense~$(**)$. Def\/ine the map $b\colon {\mathcal B}_1 \rightarrow {\mathcal R}_2$ by its action on the
generators:
\begin{gather*}
b(x) = T, \qquad b(H) = n, \qquad
b(B)= n[(n-1)(n-2)+ \alpha(n-1) + \beta]T^{-1}.
\end{gather*}
The algebra ${\mathcal B}_2$ is the image of ${\mathcal B}_1$ under the map $b$.

Set $L:= H$. Then
\begin{gather}
	H S_n(x) = n S_n(x), \label{diff-eq2}\\
	x S_n(x) = S_{n+1}(x), \label{rec-rel2} \\
	B S_n(x) = \left[\binom{n}{3} + \alpha \binom{n}{2} + \beta n\right] S_{n-1}(x). \nonumber 
\end{gather}

We see that \eqref{diff-eq2} and \eqref{rec-rel2} are providing an instance of a~dif\/ferential-dif\/ference bispectral pair of operators. Namely, we get that~\eqref{diff-eq2} guarantees that~$\psi(x,n)$ is an eigenfunction of the dif\/ferential operator~$L$, and~\eqref{rec-rel2} shows that the same function is an eigenfunction of the dif\/ference operator~$T$.

Using the above notation and following the scheme of Section~\ref{section3}, we now set up an appropriate bispectral problem.

Consider a polynomial $q(X)$ and for simplicity we assume that $q(X)$ has no constant term. Def\/ine an automorphism of the algebra ${\mathcal B}_1$:
\begin{gather*}
\sigma(A) = e^{\operatorname{ad}_{q(B)}}(A),\qquad A \in {\mathcal B}_1.
\end{gather*}
Let us compute the action of~$\sigma$. Put $R = 3(2H +1)+ 2\alpha$.

\begin{Lemma}\label{gen2}
$\sigma$ acts on the generators as follows
\begin{gather*}
\sigma(x) = x + Wq'(B) +Rq''(B)B/2 + q'''(B)B^2 \\
\hphantom{\sigma(x) =}{} + \big\{3q''(B)B +(R +6)q'(B)/2+ q'^2(B)B \big\}Bq'(B),\\
\sigma(H) = H + q'(B)B, \qquad \sigma(B) = B.
\end{gather*}
\end{Lemma}
\begin{proof}
To prove the f\/irst identity we will need to compute the commutator $[B^m, x]$.
First we have
\begin{gather*}
[B^m, x] = \sum_{j=0}^{m-1}B^j W B^{m-1 -j}.
\end{gather*}
Notice that $B W= WB + RB$. 
We also need $B^jR = (R +6j)B^j$,
from which we f\/ind
\begin{gather*}
B^jW = W B^j + \sum_{k=0}^{j-1}B^kRB^{j-k} = \left(W + \sum_{k=0}^{j-1}(R+6k)\right)B^j = [W +jR +3j(j-1)]B^j.
\end{gather*}
Hence
\begin{gather*}
[B^m, x] = \sum_{j=0}^{m-1} [W +jR +3j(j-1)] B^{m-1}.
\end{gather*}
Then using the formulas for sums of powers we transform this expression into
\begin{gather*}
[B^m, x] = [mW +m(m-1)R /2 +m(m-1)(m-2)] B^{m-1},
\end{gather*}
which shows that
\begin{gather*}
[q(B), x] = Wq'(B) +Rq''(B)B/2 + q'''(B)B^2 .
\end{gather*}
Using the commutation relation $[B,H] = B$ we compute
\begin{gather*}
[q(B), H] = q'(B)B
\end{gather*}
and
\begin{gather*}
[q(B), W] = 3/2 q''(B)B^2 +Rq'(B)B.
\end{gather*}
This gives that
\begin{gather*}
\operatorname{ad}^2_q(x) = \{6q''(B)B +(R +6)q'(B)\}Bq'(B).
\end{gather*}
Finally we f\/ind
\begin{gather*}
\operatorname{ad}^3_q(x) = 6q'^3(B)B^2, \qquad \operatorname{ad}^4_q(x) = 0.
\end{gather*}
This gives the f\/irst identity. The other two are obvious.
\end{proof}

The lemma shows that $\operatorname{ad}_{q(B)}$
acts locally nilpotently on the algebra ${\mathcal B}$, i.e., when applied to any element $A\in {\mathcal B}$, the series
\begin{gather*}
\sum_{k=0}^{\infty} \frac{\operatorname{ad}^k_{q(B)}(A)}{k!}
\end{gather*}
has only f\/inite number of terms.

Let us def\/ine a new map $b'\colon {\mathcal B}_1\to {\mathcal B}_2$ given by $b':=b\circ \sigma^{-1}$ and the
new function $\psi_1(x,n)$ def\/ined via
\begin{gather*}
\psi_1(x,n) := P_n^q(x): = e^{q(B)}x^n = x^n + \sum_{k=1}^{\infty}\frac{q^k(B) x^n}{k!}.
\end{gather*}
It is clear that the latter series contains only a f\/inite number of terms as the operator~$q(B)$ reduces the degrees of the polynomials. Hence $\psi_1(x,n)$ is a polynomial in~$x$. (Here we use that~$q$ has no constant term.)
Def\/ine the operator $L_1: = \sigma(L)$. One can easily see that
\begin{gather*}
L_1 = x\partial + q'(B)B.
\end{gather*}
In order to simplify the notations let us put
\begin{gather*}
Q_0(B) = -q'''(B)B^2 + \big\{2q''(B)B +3q'(B)- q'^2(B)B \big\}Bq'(B),\\
Q_1(B) = \frac{\big(q'^2(B) - q''(B)\big)B}{2}.
\end{gather*}
\begin{Lemma}\label{lm:new}
	The new bispectral anti-involution $b'$ satisfies:
\begin{gather*}
b'(B) = n[(n-1)(n-2)+ \alpha(n-1) + \beta] T^{-1}, \\
b'(x) = T - q'(b(B)) \big(3n^2 +2\alpha n +\beta\big) + Q_1(b(B))(6n +3 + \alpha) + Q_0(b(B)),\\
b'(L_1) = n.
	\end{gather*}
\end{Lemma}

\begin{proof} Repeats the proof of Lemma~\ref{2new-b} and will be omitted.
\end{proof}

We formulate the main results about the polynomial system $P_n^q(x)$ in the next theorem.

\begin{Theorem}
	The polynomials $P_n^q(x)$ have the following properties:
\begin{enumerate}\itemsep=0pt	
\item[$(i)$] they are eigenfunctions of the differential operator
	\begin{gather*}
		L_1 := q'(B)B + x\partial; 
	\end{gather*}
	
\item[$(ii)$] they satisfy the recursion relation
\begin{gather}
xP^q_n(x) = P^q_{n+1} + \big\{{-} 3q'(b(B))\big(3n^2 +2\alpha n +\beta\big)
+ Q_1(b(B))(6n +3 + \alpha) \nonumber\\
\hphantom{xP^q_n(x) =}{} +Q_0(b(B))\big\}P^q_n(x); \label{new-rec3}
	\end{gather}
	
\item[$(iii)$] the following differentiation formula $($lowering operator$)$	
	\begin{gather*}
		B P^q_n(x) = n[(n-1)(n-2)+ \alpha(n-1) + \beta] P^q_{n-1} 
	\end{gather*}
	holds.
\end{enumerate}
\end{Theorem}

Proof repeats that of Theorem~\ref{mult-App}.

\section{Examples}\label{section6}

In the examples below we construct the dif\/ferential operator $L_1$ that has the vector orthogonal sets $P_n^q(x) = e^{q(B)}x^n$ as eigenfunctions. In some cases we also write explicitly the f\/inite-term recursion relation.

\begin{Example}\label{Seg} Set $q(\partial) = -1/k \partial^k$. Then we obtain
\begin{gather*}
L_1 = x\partial - \partial^k.
\end{gather*}
Observe that, for $k=2$, the operator $ -L_1$ is the standard Hermite operator. For $k> 2$, $-L_1$~appeared earlier in~\cite{ST}. The recurrence relation for $P_n(x)$ is given by~\eqref{new-rec}.
In the present case it reads
\begin{gather*}
xP_n(x) = P_{n+1}(x) + \frac{n!}{(k-1)!} P_{n-k+1}(x),
\end{gather*}
which agrees with the results in~\cite{ST}. In general all the examples of Section~\ref{section3} have been studied thoroughly starting with the classical paper~\cite{App}. In the context of vector orthogonal polynomials the results have been found in~\cite{Dou}.
\end{Example}

\begin{Example}\label{ex2-1}
Here we take $B= x\partial^2 +\beta \partial$. Let us take the polynomial $q(B) = B$.
In this case we have
\begin{gather*}
L_1 = H + q'(B)B = x\partial^2 +\beta \partial + x\partial,
\end{gather*}
which up to change of the variables is the Laguerre operator.
\end{Example}

\begin{Example}\label{ex2-2}
The simplest new example is given by the polynomial
$q(B) = B^2/2$, $B = x\partial^2 +\beta \partial$.
In this case we have
\begin{gather*}
	L_1 = H + q'(B)B = \big(x\partial^2 +\beta \partial\big)^2 + x\partial
	= x^2\partial^4 + 2x(1 + \beta)\partial^3 + \big(\beta + \beta^2\big)\partial^2 + x\partial.
\end{gather*}
The 5-terms recurrence relation reads
\begin{gather*}
xP_n(x) = P_{n+1} - n(n-1 +\beta)(2n - 1 +\beta)P_{n-1} + \frac{n!}{3!}\binom{n-1 +\beta}{3} P_{n-3}.
\end{gather*}
Here are the f\/irst few polynomials
\begin{gather*}
P_0(x) = 1, \qquad P_1(x) = x, \qquad P_2(x) = x^2 +\beta(1 +\beta),\\
 P_3(x) = x^3 + 3 (1 +\beta)(2 +\beta)x, \qquad \ldots.
\end{gather*}
\end{Example}

\begin{Example}\label{ex3-1}
Let us take the polynomial
$q(B) = B $, $B = x^2\partial^3 + \alpha x\partial^2 + \beta \partial $. Then
\begin{gather*}
L_1 = H + q'(B)B = x^2\partial^3 + \alpha x\partial^2 + \beta \partial + x\partial.
\end{gather*}
The corresponding polynomials satisfy the 4-term recurrence relation
\begin{gather*}
	xP^q_n(x) = P^q_{n+1} +(3-2\alpha)P^q_n(x) - 3n[(n-1)(n-2 + \alpha) + \beta] P^q_{n-1}\\
\hphantom{xP^q_n(x) =}{}	- n(n-1)[(n-1)(n-2 +\alpha)+ \beta][(n-2)(n-3 + \alpha) + \beta] P^q_{n-2}.
\end{gather*}

This example has been produced in~\cite{BCD2} within the context of vector orthogonal polynomials which are generalized hypergeometric functions.
\end{Example}

\begin{Example}\label{ex3-2}
Let us take the polynomial
$q(B) = B^2/2 $, $B = x^2\partial^3$. Then
\begin{gather*}
L_1 = H + q'(B)B = \big(x^2\partial^3\big)^2 + x\partial = x^4 \partial^6 + 6 x^3 \partial^5 + 6 x^2 \partial^4 + x\partial.
\end{gather*}
This example seems to be new as most of the examples in Sections~\ref{section4} and~\ref{section5}.

The polynomials satisfy a 7-term recurrence relation which we skip as it is not much simpler than the corresponding general formula~\eqref{new-rec3}. Here are the f\/irst few polynomials:
\begin{gather*}
P_0(x) = 1,\qquad P_1(x) = x, \qquad P_2(x) = x^2, \\
P_3(x) = x^3 + 3\cdot 2^3x, \qquad P_4(x) = x^4 + 4\cdot 3^3\cdot 2^2x^2, \qquad \ldots.
\end{gather*}
\end{Example}

\subsection*{Acknowledgements}

The author is deeply grateful to Boris Shapiro for showing and discussing some examples of systems of polynomials studied here and in particular the examples from~\cite{ST}. Without this probably the project would have never be even started. Also Plamen Iliev made valuable comments on some results generalizing Laguerre polynomials, which helped me to place my work among the rest of the research. Milen Yakimov pointed out several errors. Many thanks go to the referees and the editor who pointed out a number of incorrect formulas and misprints and thus helped me to improve considerably the initial text.
The author is grateful to the Mathematics Department of Stockholm University for the hospitality in April 2015.
Last but not least I am extremely grateful to Professors T.~Tanev and K.~Kostadinov, and Mrs.\ Z.~Karova from the Bulgarian Ministry of Education and Science and Professor P.~Dolashka, who helped me in the dif\/f\/icult situation when I~was sacked by Sof\/ia university in violations of the Bulgarian laws\footnote{See EMS Newsletter, p.~71: \url{http://www.ems-ph.org/journals/newsletter/pdf/2015-12-98.pdf}.}.

\pdfbookmark[1]{References}{ref}
\LastPageEnding

\end{document}